\documentclass[onefignum,onetabnum]{siamart190516}

\newcommand{\be}{\begin{equation}}
\newcommand{\ee}{\end{equation}}
\newcommand{\n}{{\bf n}}

\newcommand{\x}{{\bf x}}
\newcommand{\y}{{\bf y}}
\newcommand{\avar}{{\bf a}}
\newcommand{\e}{{\bf e}}
\newcommand{\V}{{\bf V}}
\newcommand{\X}{{\bf X}}
\newcommand{\R}{\mathbb{R}}
\newcommand{\Z}{\mathbb{Z}}

\newcommand{\T}{\mathbb{T}}
\newcommand{\Lvar}{\mathcal{L}}
\newcommand{\I}{\mathcal{I}}

\usepackage{graphicx,amsfonts}

\newsiamremark{remark}{Remark}

\headers{Front Speeds of G-equation in ABC and Kolmogorov Flows}{Chou Kao, Yu-Yu Liu and Jack Xin}

\title{A Semi-Lagrangian Computation of Front Speeds of G-equation in ABC and Kolmogorov Flows with Estimation via Ballistic Orbits
\thanks{Submitted to the editors DATE.
\funding{The first author is partially supported by NCKU graduate fellowship. The second author is partially supported by MOST grant of Taiwan 104-2115-M-006-013-. The third author is partially supported by NSF grants DMS-1924548 and DMS-1952644.}}}

\author{Chou Kao
\thanks{Department of Mathematics,
National Cheng-Kung University, Tainan 70101, Taiwan
(\email{L18101030@gs.ncku.edu.tw}).}
\and Yu-Yu Liu
\thanks{Corresponding author. Department of Mathematics,
National Cheng-Kung University, Tainan 70101, Taiwan
(\email{yuyul@ncku.edu.tw}).}
\and Jack Xin
\thanks{Department of Mathematics,
University of California, Irvine, Irvine, CA 92697, USA
(\email{jxin@math.uci.edu}).}}

\begin{document}

\maketitle

\begin{abstract}
The Arnold-Beltrami-Childress (ABC) flow and the Kolmogorov flow are three dimensional periodic divergence free velocity fields that exhibit chaotic streamlines. We are interested in front speed enhancement in G-equation of turbulent combustion by large intensity ABC and Kolmogorov flows. We give a quantitative construction of the ballistic orbits of ABC and Kolmogorov flows, namely those with maximal large time asymptotic speeds in a coordinate direction. Thanks to the optimal control theory of G-equation (a convex but non-coercive Hamilton-Jacobi equation), the ballistic orbits serve as admissible trajectories for front speed estimates. To study the tightness of the estimates, we compute the front speeds of G-equation based on a semi-Lagrangian (SL) scheme with Strang splitting and weighted essentially non-oscillatory (WENO) interpolation. Time step size is chosen so that the Courant number grows sublinearly with the flow intensity. Numerical results show that the front speed growth rate in terms of the flow intensity may approach the analytical bounds from the ballistic orbits.
\end{abstract}



\section{Introduction}
The study of transport phenomena in three dimensional fluid flows is a challenging problem, due in part to the presence of chaos and the high computational costs in resolving small scales, \cite{W85,MK99,P00,X00,X09} and references therein. 
In this paper, we consider the Arnold-Beltrami-Childress (ABC) flow \cite{A65,DFG86}
\be\label{ABC}
\V_1(x,y,z)=\left\langle
\sin z+\cos y, \sin x+\cos z, \sin y+\cos x
\right\rangle.
\ee
and the Kolmogorov flow \cite{CG95} (or Archontis flow \cite{A11})
\be\label{K}
\V_2(x,y,z)=\left\langle
\sin z, \sin x,\sin y
\right\rangle.
\ee
While periodic in $(2\pi\T)^3$, these flows are well-known for exhibiting chaotic streamlines. They have been studied in many contexts, including the electromagnetic conductivity in kinematic dynamo problem, the traveling wave speed in reaction-diffusion-advection equation and the eddy diffusivity \cite{BCVV95,CG95,G12,GP92,SXZ13,WXZ}.

Denote by $\X:\R\to\R^3$ a Lagrangian trajectory of ABC or Kolmogorov flow satisfying: $\dot{\X}(\cdot)=\V(\X(\cdot))$. In search of the ballistic orbits in say $x$-direction ($y$- and $z$- directions are similar), let the trajectory start from $yz$-plane and evaluate its large time asymptotic speed in $x$-direction as ($\e_1 = \langle 1,0,0\rangle$):
\be\label{Xbar}
\bar{x}=\lim_{t\to\infty}{\X(t)\cdot\e_1\over t},\,\,\, \X(0)=\langle0,y,z \rangle.
\ee
See \cref{Xmap}. The orbits $\X(t)$ are generated on a 800$\times$800 mesh of $\langle y,z\rangle\in(2\pi\T)^2$ by ODE solver in MATLAB (ode113), and the propagation speeds $\bar{x}$ are evaluated at $t=1000$. For ABC flow, $\bar{x}$ reaches maximum when $\X(0)\approx\langle0,5.942,1.571\rangle$ (accurate to three decimal places); for Kolmogorov flow, $\bar{x}$ reaches maximum when $\X(0)\approx\langle0,0.029,1.571\rangle$. It turns out that these orbits with maximum asymptotic speeds are periodic (modulo $2\pi$) in $x$-direction, that is, there exists $\tau>0$ such that $\X(\cdot+\tau)=\X(\cdot)+2\pi\cdot\e_1$. See \cref{Xorbit}. Also the orbits with minimum asymptotic speeds are periodic in negative $x$-direction: $\X(\cdot+\tau)=\X(\cdot)-2\pi\cdot\e_1$.

The periodic orbit of ABC flow was first proved to exist in \cite{XYZ16}. The authors found an orbit that starts from the line segment $\{x=-\pi/2,y=0,z\in[0,\pi/2]\}$ and passes through the line segment $\{x=0,y\in[-\pi/2,3\pi/2],z=\pi/2\}$. Thanks to the symmetries of ABC flow, such orbit also inherits certain symmetries and therefore is periodic in $x$-direction. For Kolmogorov flow, we found numerically that the periodic orbit starts from line segment $\{x=0,y\in[0,\pi/2],z=\pi/2\}$ and passes through line segment $\{x=\pi/2,y=\pi,z\in[\pi,3\pi/2]\}$. See \cref{prop:1} and \cref{prop:2} for precise statements.

In turbulent combustion theory, G-equation is a front propagation model of thin flames \cite{W85,P00}:
\be\label{Geq}
{\partial G\over\partial t}+\V(\x)\cdot\nabla G+|\nabla G|=0.
\ee
Formulated by level set method, the flame front $\{G(\x,t)=0\}$ moves in the laminar velocity $\n=\nabla G/|\nabla G|$ due to fuel combustion along with the flow velocity $\V(\x)$ due to fuel convection. In three dimensional space, let the initial flame front be the $yz$-plane: 
\be\label{IC}
G(\x,0)=\x\cdot\e_1,\,\,\,\x\in\R^3.
\ee
Eventually the flame front propagates in $x$-direction at the so called 
{\it turbulent flame speed}:
\be\label{sTa}
s_T :=\lim_{t\to\infty}-{G(\x,t)\over t}
\ee
where convergence holds for all $\x$ and $s_T$ is independent of $\x$.
One fundamental issue in turbulent combustion theory is front speed enhancement due to fluid convection. In G-equation model, let the flow velocity be ABC flow (\ref{ABC}) with intensity $A>0$: 
\be\label{AV}
\V(\x)=A\cdot\V_1(\x)=A\cdot\left\langle
\sin z+\cos y, \sin x+\cos z, \sin y+\cos x
\right\rangle
\ee
or Kolmogorov flow (\ref{K}): $\V(\x)=A\cdot\V_2(\x)=A\cdot\left\langle
\sin z, \sin x, \sin y\right\rangle$. We would like to study the growth rate of turbulent flame speed with respective to the flow intensity: $s_T(A)$ as a function of $A$. In the case of two dimensional cellular flow $\V(x,y)=A\!\cdot\!\langle-\sin x\cos y,\cos x\sin y\rangle$, the growth rate of turbulent flame speed is given by $s_T(A)=O(A/\log A)$ \cite{O00,XY13}. Using the optimal control theory of Hamilton-Jacobi-Bellman (HJB) equation, the ballistic orbits are chosen as admissible trajectories to obtain the upper and lower bounds of turbulent flame speeds. See \cref{thm:1}.

Discretized as a monotone and consistent numerical Hamiltonian, finite difference computation of G-equation has been quite successful in two dimensional space \cite{OF02,LXY13}. When it comes to three dimensional space however, the computational cost increases considerably in large flow intensity regime. Specifically, the Courant number as well as the constraint of time step size (CFL condition, assuming $\Delta x=\Delta y=\Delta z$) reads 
\be\label{CFL_FD}\begin{array}{ll}
(6A+\sqrt{3})\cdot\Delta t/\Delta x\leq 1 & \mbox{(ABC flow)}\\
(3A+\sqrt{3})\cdot\Delta t/\Delta x\leq 1 & \mbox{(Kolmogorov flow)}
\end{array}.
\ee
Therefore it is desirable to consider other numerical methods when the flow intensity $A$ is large. 

Semi-Lagrangian (SL) scheme was first introduced as first-order approximation of scalar convection equation (also called the Courant-Isaacson-Rees scheme \cite{CIR52}). Further developed with many techniques such as dimensional splitting or higher order interpolation, semi-Lagrangian scheme has been very popular in weather forecast modeling and many other multidimensional atmospheric problems \cite{SC91}. 
As the semi-Lagrangian scheme being applied on the advection term in G-equation, it remains to discretize the laminar term. In \cite{A10}, the solution is considered smooth, and the laminar velocity is incorporated into the flow velocity for higher order approximation. In \cite{FF02}, the laminar term is discretized by Hopf-Lax formula, and the solution is evaluated through function minimization. In our present work, thanks to operator splitting, the flow velocity is discretized by semi-Lagrangian scheme, and the function is evaluated by WENO interpolation \cite{CFR05,FF13}; the laminar velocity is discretized by finite difference method, and the derivatives are evaluated by HJ WENO scheme \cite{JP00,OF02,S09}.

The rest of paper is organized as follows. In \cref{sec:2}, we find the ballistic orbits of ABC and Kolmogorov flows numerically and verify that these orbits are periodic (modulo 2$\pi$) in $x$-direction. In \cref{sec:3}, we present the control formulation of G-equation and obtain the estimates of turbulent flame speeds. In \cref{sec:4}, we provide the semi-Lagrangian discretization of G-equation and the numerical results of turbulent flame speeds. In \cref{sec:5}, we conclude the paper with comments and future works.

\begin{figure}
\begin{center}
\includegraphics[width=\textwidth]{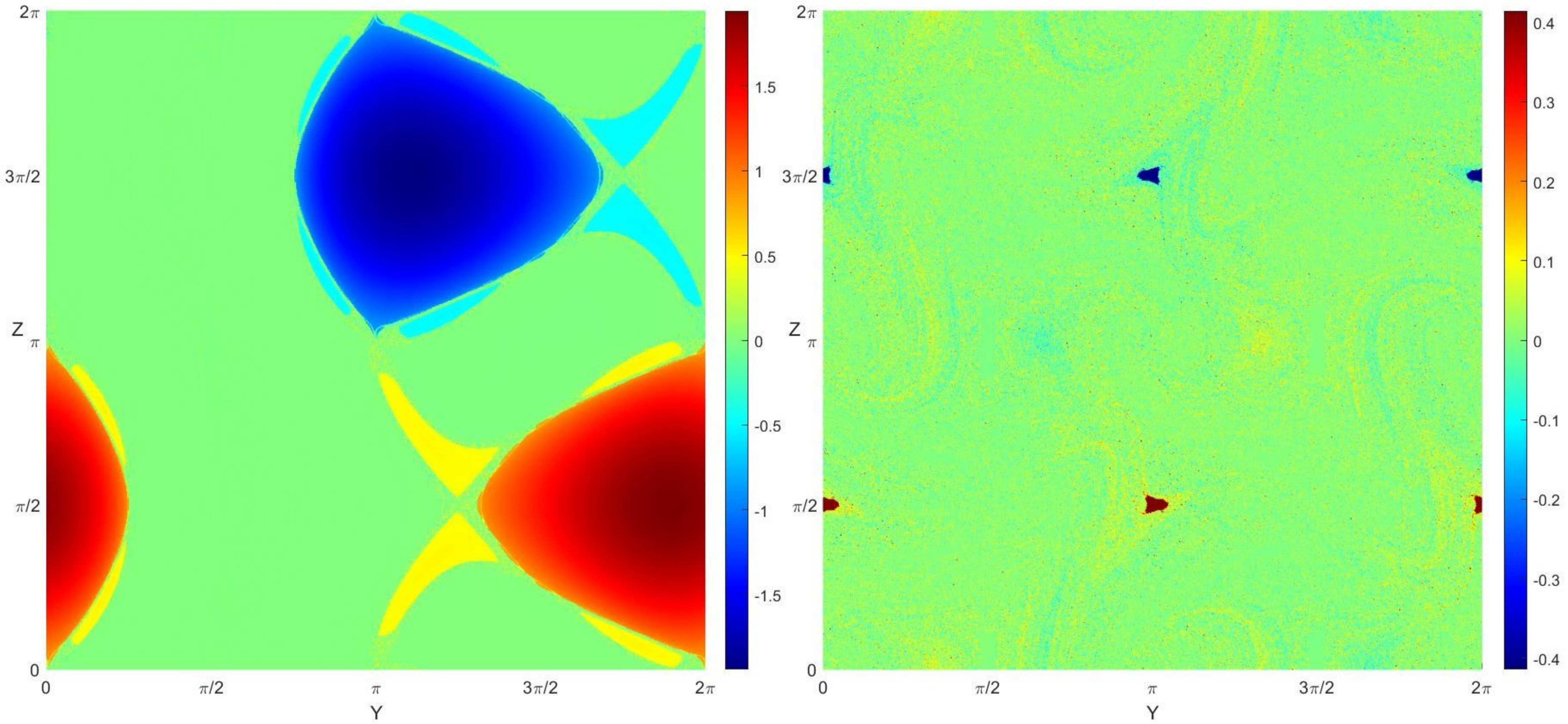}
\end{center}
\caption{Approximate asymptotic speeds of Lagrangian orbits in $x$-direction evaluated by (\ref{Xbar}). Left: ABC flow. Right: Kolmogorov flow.}
\label{Xmap}
\end{figure}

\begin{figure}
\includegraphics[width=\textwidth]{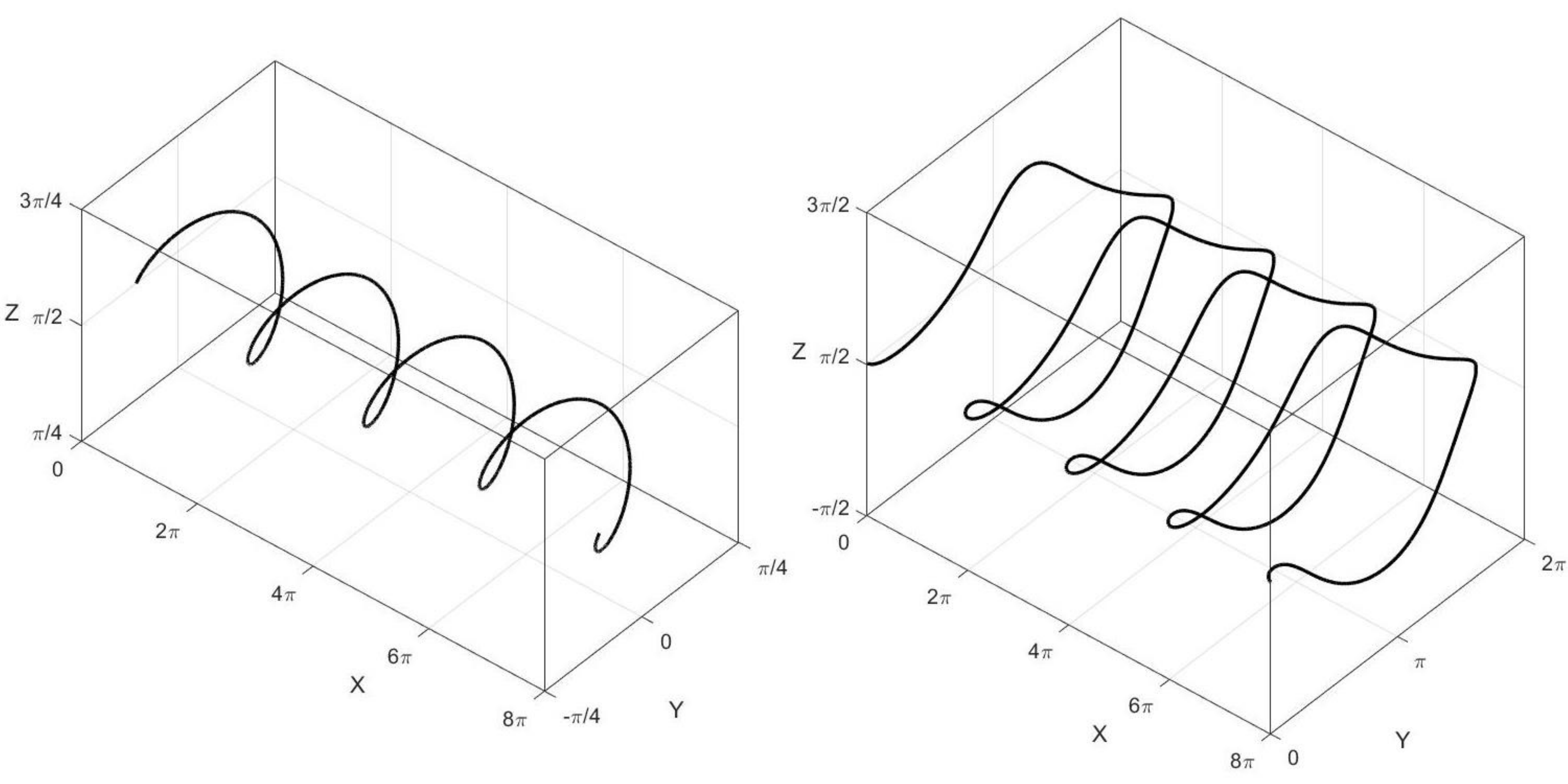}
\caption{Ballistic orbits periodic (modulo 2$\pi$) in $x$-direction. Left: ABC flow. Right: Kolmogorov flow.}
\label{Xorbit}
\end{figure}

\section{Ballistic Orbits of ABC and Kolmogorov Flows}
\label{sec:2}

We restate the periodic orbit of ABC flow in \cite{XYZ16} with more numerical description and the symmetry argument, then we present the periodic orbit of Kolmogorov flow in the same fashion. Recall $\odot$ denotes the Hadamard (element-wise) product of vectors.

\begin{proposition}\label{prop:1}
There exists an orbit $\X_1$ of the ABC flow such that
$$\X_1(0)=\langle0,-a_1,\pi/2\rangle,\, 
\X_1(\tau_1/4)=\langle\pi/2,0,\pi/2+b_1\rangle,$$
where $a_1\approx0.341$, $b_1\approx0.341$ and $\tau_1\approx3.235$. Then $\X_1$ is a ballistic orbit of the ABC flow in $x$-direction with period $\tau_1$. 

Also the orbit $\X_{-1}$ with $\X_{-1}(0)=\langle0,\pi+a_1,-\pi/2\rangle$ is a ballistic orbit periodic in negative $x$-direction.
\end{proposition}

\begin{proof}
(a) See left of \cref{orbitPf}. If $\X_1(0)=\langle 0,-0.342,\pi/2\rangle$, then $\X_1(t)$ passes through the face $\{x=\pi/2,y<0,z\in[\pi/2,\pi]\}$; if $\X_1(0)=\langle 0,-0.341,\pi/2\rangle$, then $\X_1(t)$ passes through the face $\{x=\pi/2,y>0,z\in[\pi/2,\pi]\}$. Therefore there exists $a_1\in(0.341,0.342)$ such that if $\X_1(0)=\langle0,-a_1,\pi/2\rangle$ then $\X_1(t)$ passes through the edge $\{x=\pi/2,y=0,z\in[\pi/2,\pi]\}$ between the two faces. 


(b) Observe the symmetry of the flow about axis $\{x=\pi/2,y=0,z\in\R\}$:
$$\V_1(\pi-x,-y,z)
=\langle1,1,-1\rangle\odot\V_1(x,y,z).$$
Since $\X_1(\tau_1/4)$ lies on the axis, we have the symmetry of the orbit
$$\X_1(\tau_1/2)-\X_1(\tau_1/4)
=\langle1,1,-1\rangle\odot
(\X_1(\tau_1/4)-\X_1(0))$$
and therefore $\X_1(\tau_1/2)=\langle\pi,a_1,\pi/2\rangle$.

Observe the symmetry of the flow about axis $\{x=\pi,y\in\R,z=\pi/2\}$:
$$\V_1(2\pi-x,y,\pi-z)
=\langle1,-1,1\rangle\odot\V_1(x,y,z).$$
Since $\X_1(\tau_1/2)$ lies on the axis, we have the symmetry of the orbit
$$\X_1(\tau_1)-\X_1(\tau_1/2)
=\langle1,-1,1\rangle\odot
(\X_1(\tau_1/2)-\X_1(0))$$
and therefore $\X_1(\tau_1)=\langle2\pi,-a_1,\pi/2\rangle=\X_1(0)+2\pi\cdot\e_1$. 

(c) Observe the symmetry of the flow:
$$\V_1(-x,\pi-y,-\pi+z)=\langle-1,-1,1\rangle\odot\V_1(x,y,z).$$ 
Since $\X_1(0)$ and $\X_{-1}(0)$ satisfy the symmetry condition, we have the symmetry of the orbit
$$\X_{-1}(\tau_1)-\X_{-1}(0)=\langle-1,-1,1\rangle\odot
(\X_1(\tau_1)-\X_1(0))$$
and therefore $\X_{-1}(\tau_1)=\langle-2\pi,\pi+a_1,-\pi/2\rangle=\X_{-1}(0)-2\pi\cdot\e_1$.
\end{proof}

\begin{remark}
The symmetry condition
$$\V_1(\pi/2-x,\pi/2-z,\pi/2-y)=\sigma\circ\V_1(x,y,z)$$ with  $\sigma:(x,y,z)\mapsto(x,z,y)$ further implies $a_1=b_1$.
\end{remark}

\begin{proposition}\label{prop:2} 
There exists an orbit $\X_2$ of the Kolmogorov flow such that
$$\X_2(0)=\langle0,a_2,\pi/2\rangle,\,
\X_2(\tau_2/4)=\langle\pi/2,\pi,\pi+b_2\rangle,$$ 
where $a_2\approx0.029$, $b_2\approx0.602$ and $\tau_2\approx15.156$. Then $\X_2$ is a ballistic orbit of the Kolmogorov flow in $x$-direction with period $\tau_2$.

Also the orbit $\X_{-2}$ with $\X_{-2}(0)=\langle0,-a_2,-\pi/2\rangle$ is a ballistic orbit periodic in negative $x$-direction.
\end{proposition}

\begin{proof}
(a) See right of \cref{orbitPf}. If $\X_2(0)=\langle 0,0.029,\pi/2\rangle$, then $\X_2(t)$ passes through the face $\{x>\pi/2,y=\pi,z\in[\pi,3\pi/2]\}$; if $\X_2(0)=\langle0,0.03,\pi/2\rangle$, then $\X_2(t)$ passes through the face $\{x<\pi/2,y=\pi,z\in[\pi,3\pi/2]\}$. Therefore there exists $a_2\in(0.029,0.03)$ such that if $\X_2(0)=\langle0,a_2,\pi/2\rangle$ then $\X_2(t)$ passes through the edge $\{x=\pi/2,y=\pi,z\in[\pi,3\pi/2]\}$ between the two faces.


(b) Observe the symmetry of the flow about axis $\{x=\pi/2,y=\pi,z\in\R\}$:
$$\V_2(\pi-x,2\pi-y,z)
=\langle1,1,-1\rangle\odot\V_2(x,y,z).$$
Since $\X_2(\tau_2/4)$ lies on the axis, we have the symmetry of the orbit
$$\X_2(\tau_2)-\X_2(\tau_2/4)
=\langle1,1,-1\rangle\odot
(\X_2(\tau_2/4)-\X_2(0))$$
and therefore $\X_2(\tau_2/2)=\langle\pi,2\pi-a_2,\pi/2\rangle$.

Observe the symmetry of the flow about axis $\{x=\pi,y\in\R,z=\pi/2\}$:
$$\V_2(2\pi-x,y,\pi-z)
=\langle1,-1,1\rangle\odot\V_2(x,y,z).$$
Since $\X_2(\tau_2/2)$ lies on the axis, we have the symmetry of the orbit
$$\X_2(\tau_2)-\X_2(\tau_2/2)
=\langle1,-1,1\rangle\odot
(\X_2(\tau_2/2)-\X_2(0))$$
and therefore $\X_2(\tau_2)=\langle2\pi,a_2,\pi/2\rangle=\X_2(0)+2\pi\cdot\e_1$.

(c) Observe the symmetry of the flow about point $\langle0,0,0\rangle$:
$$\V_2(-x,-y,-z)=\langle-1,-1,-1\rangle\odot\V_2(x,y,z).$$ 
Since $\X_2(0)$ and $\X_{-2}(0)$ are symmetric about the point, we have the symmetry of the orbit
$$\X_{-2}(\tau_2)-\X_{-2}(0)=\langle-1,-1,-1\rangle\odot
(\X_2(\tau_2)-\X_2(0))$$ 
and therefore $\X_{-2}(\tau_2)=\langle-2\pi,-a_2,-\pi/2\rangle=\X_{-2}(0)-2\pi\cdot\e_1$.
\end{proof}

\begin{remark}
The symmetry of the flow
$$\V_2(x,\pi+y,z)=\langle1,1,-1\rangle\odot\V_2(x,y,z)$$
implies $\X_2(0)=\langle0,\pi+a_2,\pi/2\rangle$ gives another ballistic orbit periodic in $x$-direction, and  $\X_{-2}(0)=\langle0,-\pi-a_2,-\pi/2\rangle$ gives another ballistic orbit periodic in negative $x$-direction.
\end{remark}

\begin{figure}
\begin{center}
\includegraphics[width=\textwidth]{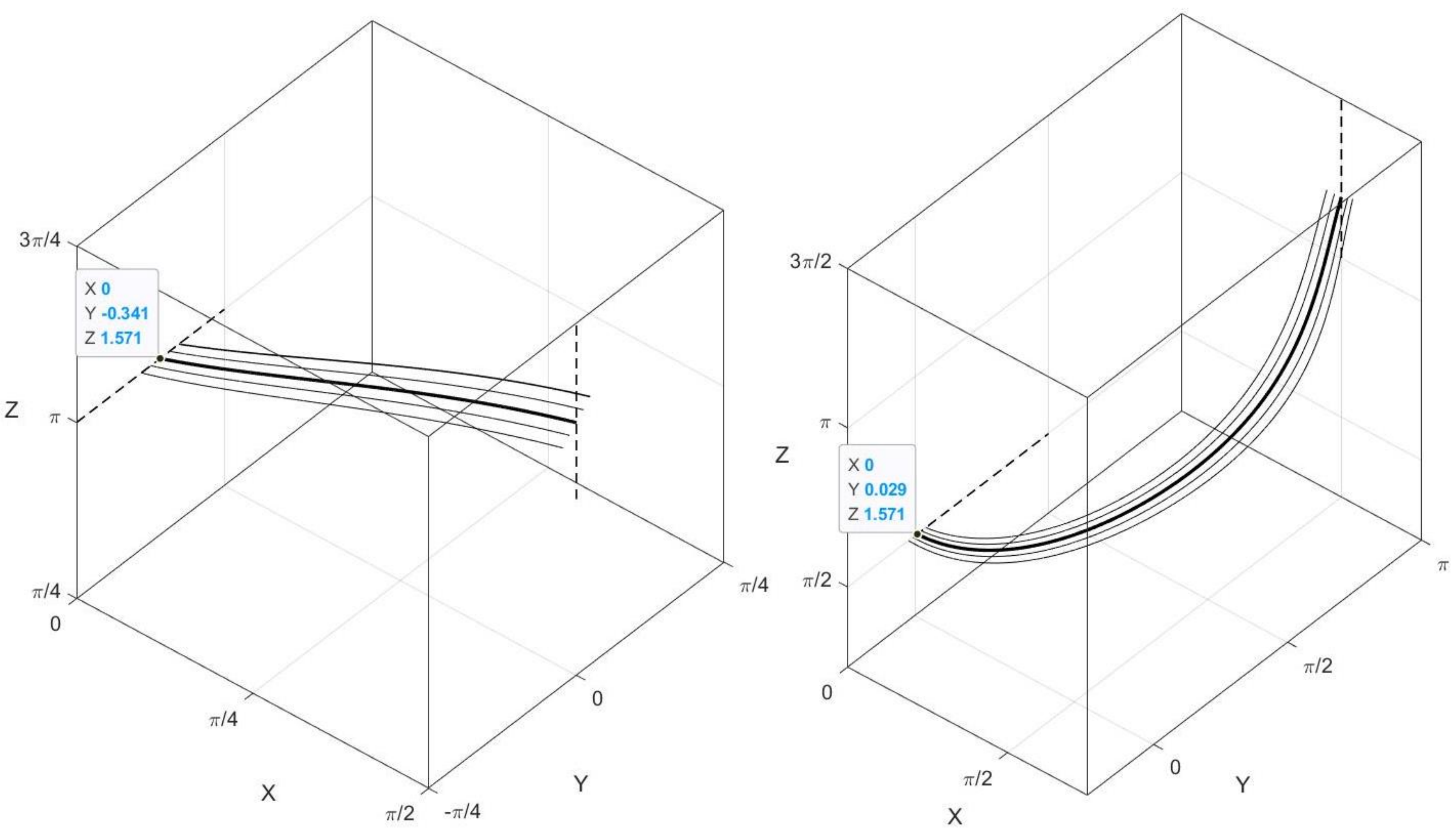}
\end{center}
\caption{Numerical validation of ballistic orbits periodic (modulo 2$\pi$) in $x$-direction. Left: ABC flow. Right: Kolmogorov flow.}
\label{orbitPf}
\end{figure}

\section{Estimates of Turbulent Flame Speeds}
\label{sec:3}

The solution of G-equation (\ref{Geq}) is given by the control representation formula:
$$G(\y,t)=\inf_{\avar(\cdot)}G(\x_{\y,\avar}(t),0),$$
where the infimum is over all admissible controls $\avar:(0,t)\to\R^3$ and the corresponding trajectories $\x=\x_{\y,\avar}:(0,t)\to\R^3$ satisfying
\be\label{ctrl}
\dot{\x}(\cdot)=\V(\x(\cdot))+\avar(\cdot),\,|\avar(\cdot)|\leq 1,\,\x(0)=\y.
\ee
Given the initial condition (\ref{IC}), the turbulent flame speed (\ref{sTa}) is equivalent to:
\be\label{sTb}
s_T=\lim_{t\to\infty}\sup_{\avar(\cdot)}{\x_{\y,\avar}(t)\cdot(-\e_1)\over t}.
\ee
As a dynamical programming problem, finding the supremum among all admissible controls in large time will suffer from the ``curse of dimensionality''. By finding an admissible trajectory that traces in negative $x$-direction as far as possible, its propagation speed gives a lower estimate of turbulent flame speed.

\begin{theorem}\label{thm:1}
Consider the G-equation (\ref{Geq}) where the flow velocity is the ABC flow or Kolmogorov flow and the flow intensity is $A>0$ (\ref{AV}). Let $s_T$ be the turbulent flame speed in $x$-direction  (\ref{sTb}), then
\be\label{sTbd}
{2\pi\over\tau_i}\cdot A+{2\pi\over\tau_i\cdot\|\V_i\|_{\infty}}\leq s_T(A)\leq \|\V_i\cdot\e_1\|_{\infty}\cdot A+1,
\ee
where $\tau_i$ is the period of the ballistic orbit of the flow $\V_i$, $i=1,2$.
\end{theorem}

\begin{proof}
For any admissible trajectories, the $x$-component of (\ref{ctrl}) implies
$$
\dot{\x}\cdot(-\e_1)\leq \|\V\cdot\e_1\|_{\infty}+1
=\|\V_i\cdot\e_1\|_{\infty}\cdot A+1.
$$ 
Therefore the upper bound follows immediately.
Recall $\X_{-i}$ denotes the ballistic orbit of flow $\V_i$ periodic in negative $x$-direction.
In (\ref{ctrl}), denote $\x_{-i}$ the admissible trajectory by choosing the initial position same to the ballistic orbit and the control to be the unit tangent of flow $\V=A\cdot\V_i$ at present position:
$$\x_{-i}(0)=\X_{-i}(0),\,\,\,\avar_{-i}(\cdot)={\V(\x_{-i}(\cdot))\over|\V(\x_{-i}(\cdot))|}={\V_i(\x_{-i}(\cdot))\over|\V_i(\x_{-i}(\cdot))|}.$$ 
Then $\x_{-i}$ is identical to $\X_{-i}$ but traced with different speed:
$$\dot{\x}_{-i}
=A\cdot\V_i(\x_{-i})+\avar_{-i}
={A\cdot|\V_i(\x_{-i})|+1\over|\V_i(\x_{-i})|}\V_i(\x_{-i})
={A\cdot|\V_i(\X_{-i})|+1\over|\V_i(\X_{-i})|}\dot{\X}_{-i}.$$
When a particle traces along the ballistic orbit $\X_{-i}$, the period is $\tau_i$ and the average speed is $2\pi/\tau_i$. For ballistic orbit $\x_{-i}$, the period becomes 
$$\tau'_{i}=\int_0^{\tau_i}{|\V_i(\X_{-i}(t)|\over A\cdot|\V_i(\X_{-i}(t)|+1}dt\leq{\|\V_i\|_{\infty}\over A\cdot\|\V_i\|_{\infty}+1}\tau_{i}$$
and the average speed
$${2\pi/\tau'_i}\geq (2\pi/\tau_i)\cdot(A+1/\|\V_i\|_{\infty})$$
gives a lower bound of turbulent flame speed.
\end{proof}

\begin{remark}
(a) For ABC flow, $\|\V_1\cdot\e_1\|_{\infty}=2$, $\|\V_1\|_{\infty}=\sqrt{6}$, and $\tau_1\approx 3.235$.  Therefore the estimate (\ref{sTbd}) reads
$$1.942\cdot A+0.793\leq s_T(A)\leq 2\cdot A+1.$$
(b) For Kolmogorov flow, $\|\V_2\cdot\e_1\|_{\infty}=1$, $\|\V_2\|_{\infty}=\sqrt{3}$, and $\tau_2\approx 15.156$.  Therefore the estimate  (\ref{sTbd}) reads
$$0.414\cdot A+0.239\leq s_T(A)\leq A+1.$$
\end{remark}

\section{Semi-Lagrangian Scheme for G-equation}
\label{sec:4}

Let G-equation (\ref{Geq}) be written in the operator form
$$
{\partial G\over\partial t}
+\Lvar_xG+\Lvar_yG+\Lvar_zG+\Lvar_eG=0,
$$
where $\Lvar_x, \Lvar_y, \Lvar_z$ are the convection terms in $x,y,z$-direction respectively and $\Lvar_e$ is the laminar term,
then the solution is presented in the semi-group form
$$
G(\cdot,t)=e^{(\Lvar_x+\Lvar_y+\Lvar_z+\Lvar_e)t}G(\cdot,0).
$$
Due to Strang splitting, its temporal approximation is given by
$$
G(\cdot,t+\Delta t)\approx 
e^{\Lvar_x{\Delta t\over2}}
e^{\Lvar_y{\Delta t\over2}}
e^{\Lvar_z{\Delta t\over2}}
e^{\Lvar_e\Delta t}
e^{\Lvar_z{\Delta t\over2}}
e^{\Lvar_y{\Delta t\over2}}
e^{\Lvar_x{\Delta t\over2}}G(\cdot,t).$$
Therefore it suffices to consider the convection equation in $x$-direction ($y$- and $z$-directions are similar)
\be\label{GeqC}
{\partial G\over\partial t}
+c(\x){\partial G\over\partial x}=0
\ee
and the eikonal equation
\be\label{GeqE}
{\partial G\over\partial t}+|\nabla G|=0.
\ee

For scalar convection equation (\ref{GeqC}), its first order semi-Lagrangian discretization is given by
$$
G^{n+1}_{i,j,k}=\I[G^n](x_i-c_{i,j,k}\Delta t,y_j,z_k),
$$
where $\I[\cdot]$ denotes the interpolation of the function. For example, if 
$$
x_i-c_{i,j,k}\Delta t\in[x_{i'},x_{i'+1}], i'\in\Z,
$$
then the function value between the grid points is evaluated by
$$
\I[G^n](x_{i'}+\lambda\Delta x,y_j,z_k)\approx (1-\lambda)G^n_{i',j,k}+\lambda G^n_{i'+1,j,k}, \lambda\in[0,1].
$$
To improve the numerical results, the characteristic curves are obtained by high accuracy solvers, and the function values are evaluated by WENO interpolation \cite{CFR05,FF13}.

For eikonal equation (\ref{GeqE}), its first order finite difference forward Euler discretization is given by
$$
{G^{n+1}_{i,j,k}-G^n_{i,j,k}\over\Delta t}+\sqrt{(D_xG^n_{i,j,k})^2+(D_yG^n_{i,j,k})^2+(D_zG^n_{i,j,k})^2}=0,
$$
where the spatial derivatives are evaluated by the Godunov flux of one-side derivatives:
$$\begin{array}{c}
(D_xG^n_{i,j,k})^2=\max\left(
\max\left({G^n_{i,j,k}-G^n_{i-1,j,k}\over\Delta x},0\right)^2,
\min\left({G^n_{i+1,j,k}-G^n_{i,j,k}\over\Delta x},0\right)^2
\right),\\
(D_yG^n_{i,j,k})^2=\max\left(
\max\left({G^n_{i,j,k}-G^n_{i,j-1,k}\over\Delta y},0\right)^2,\min\left({G^n_{i,j+1,k}-G^n_{i,j,k}\over\Delta y},0\right)^2
\right),\\
(D_zG^n_{i,j,k})^2=\max\left(
\max\left({G^n_{i,j,k}-G^n_{i,j,k-1}\over\Delta z},0\right)^2,\min\left({G^n_{i,j,k+1}-G^n_{i,j,k}\over\Delta z},0\right)^2
\right).
\end{array}
$$
To achieve higher order accuracy, the one-side spatial derivatives are evaluated by HJ WENO scheme, and the time steps are iterated by TVD (total variation diminishing) Runge-Kutta scheme \cite{JP00,OF02,S09}.

We would like to solve G-equation \cref{Geq} with planar initial condition \cref{IC} in whole space. If we write $G(\x,t)=\x\cdot\e_1+U(\x,t)$, thanks to $\V(\x)$ being periodic on  $(2\pi\T)^3$, then $U(\x,t)$ satisfies the initial value problem on periodic domain:
\be\label{IVP}
\left\{\begin{array}{l}
\displaystyle{\partial U\over\partial t}+\V(\x)\cdot(\nabla U+\e_1)+|\nabla U+\e_1|=0
\smallskip\\
U(\x,0)=0,\,\x\in(2\pi\T)^3\,,t>0
\end{array}\right..
\ee
Numerical computation of (\ref{IVP}) is performed on a cubic domain with mesh size 160$\times$160$\times$160. Note that the semi-Lagrangian scheme is unconditionally stable, therefore the constraint of  time-step size comes from the eikonal equation (\ref{GeqE}) only: 
\be\label{CFL_E}
\sqrt{3}\cdot\Delta t/\Delta x\leq1.
\ee
For the convection equation with variable velocity (\ref{GeqC}), the local truncation error increases as the flow intensity increases \cite{D10}. Therefore we shall reduce $\Delta t$ with respect to increase of $A$. For a balance between computational efficiency and accuracy, we choose the time step size as the intermediate of (\ref{CFL_FD}) and (\ref{CFL_E}) by
$$\sqrt{A+3}\cdot\Delta t/\Delta x\leq1.$$

See \cref{surf} for the level surfaces propagating in ABC flow (left) and 
Kolmogorov flow (right). Notice that while the flame front moves forward along 
the ballistic orbits of the flow in $x$-direction, 
there are tails close behind due to the flame front 
being dragged backward by the ballistic orbits in negative $x$-direction. 
The level surfaces contain sophisticated structures 
as a result from interaction with the flow. 

See \cref{sTplot} for the plot of the turbulent 
flame speed $s_T$ with respect to the flow intensity $A$ using (\ref{sTa}). 
Observe that the computed $s_T(A)$ for ABC flow indeed lies 
between the narrow gap of the upper and lower bounds. This indicates good accuracy 
of the proposed numerical scheme, and that 
the numerical diffusion is well reduced by higher order WENO schemes. 
For Kolmogorov flow, we see that the $s_T(A)$ curve  
almost attaches to the upper bound when $A$ is small. 
As the flame front is weakly corrugated in weak convection, 
front speed enhancement is driven by the shear speed $\V_2\cdot\e_1=\sin z=1$ 
along plane $z=\pi/2$. When $A$ is large, the $s_T(A)$ curve tends to be 
parallel to the lower bound. 
As the flame front is severely corrugated in strong convection, 
front speed enhancement is driven by the ballistic orbit $\X_2$ of the flow with 
asymptotic speed $2\pi/\tau_2\approx0.414$ which appears in the lower bound of \cref{thm:1}.

\begin{figure}
\includegraphics[width=\textwidth]{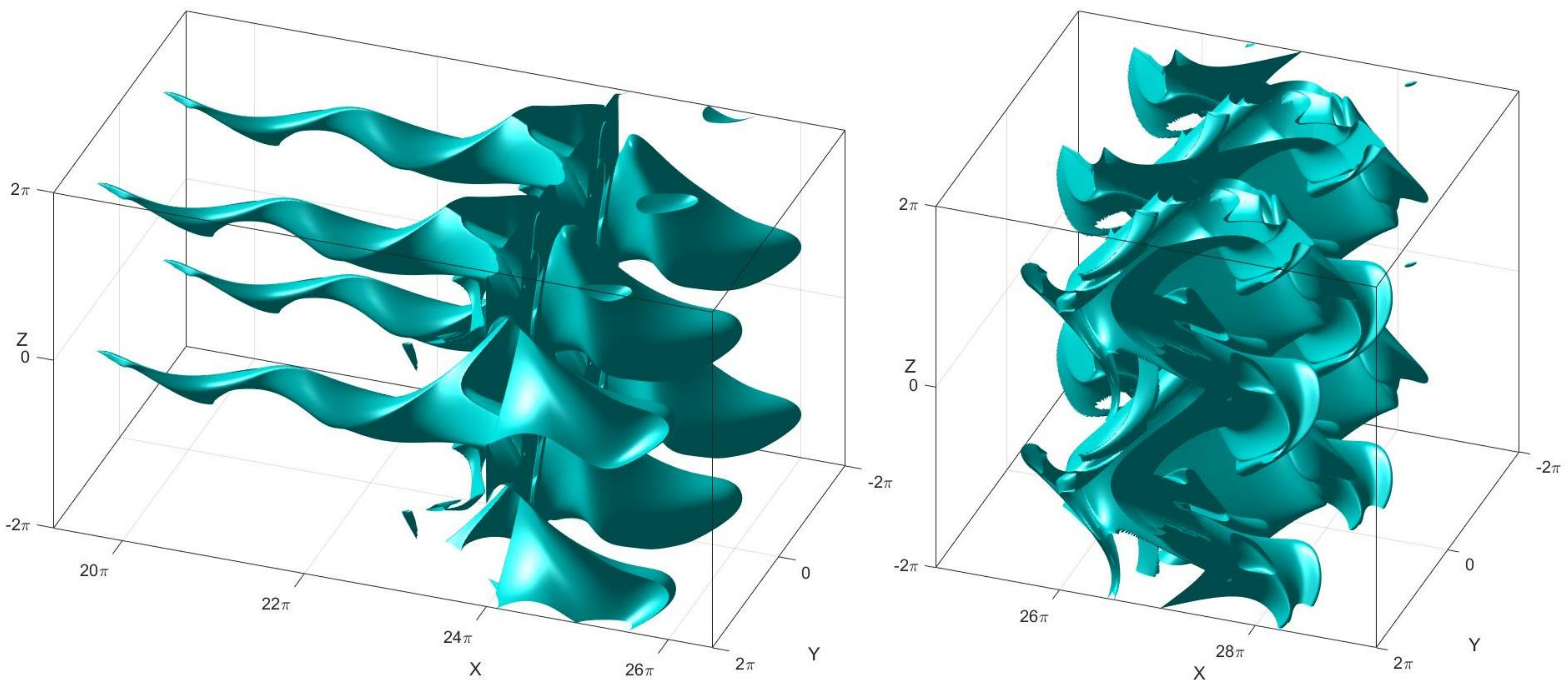}
\caption{Level surfaces of G-equation at $t=10$. Left: ABC flow with $A=4$. Right: Kolmogorov flow with $A=16$.}
\label{surf}
\end{figure}

\begin{figure}
\includegraphics[width=\textwidth]{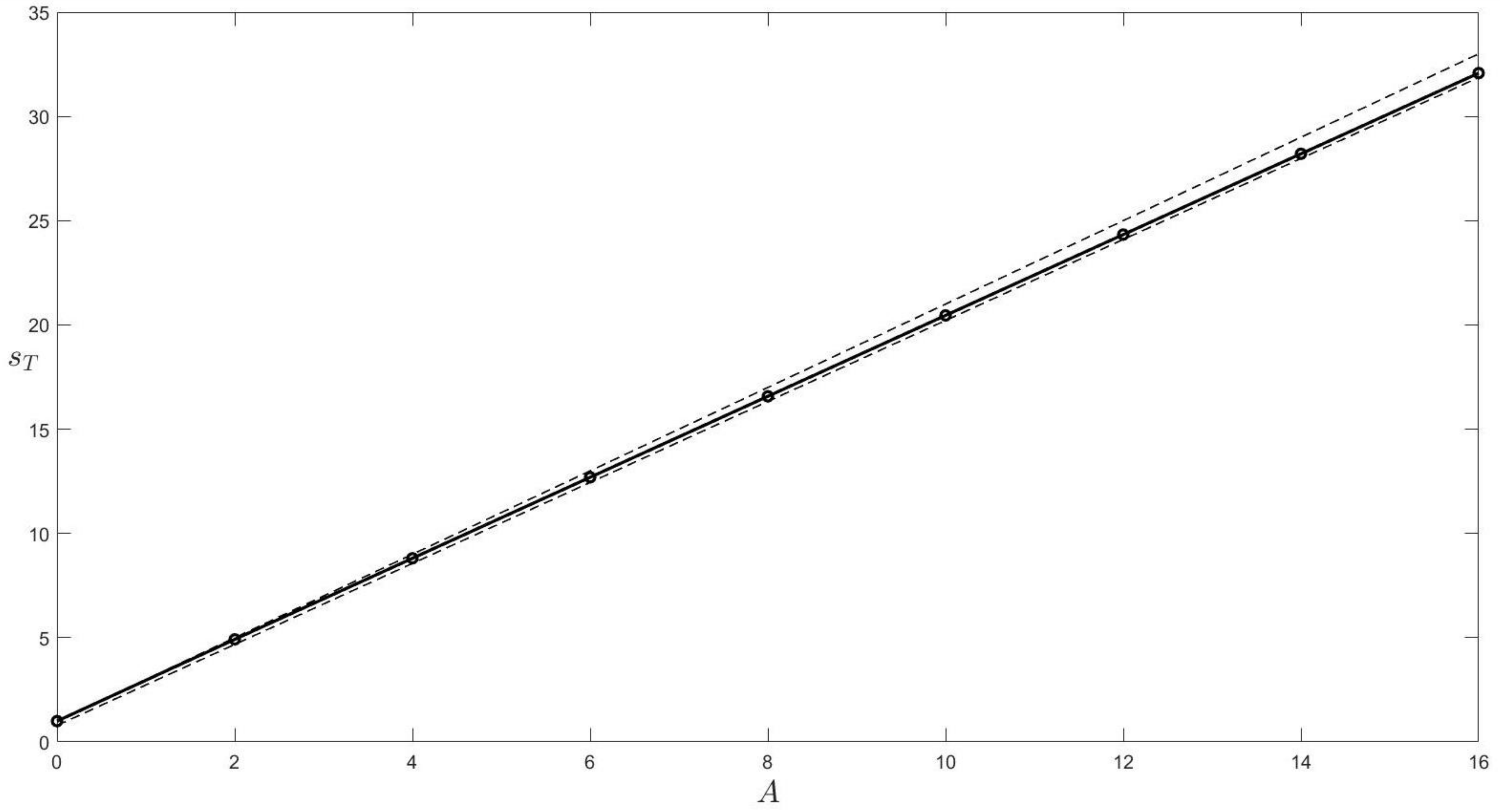}
\includegraphics[width=\textwidth]{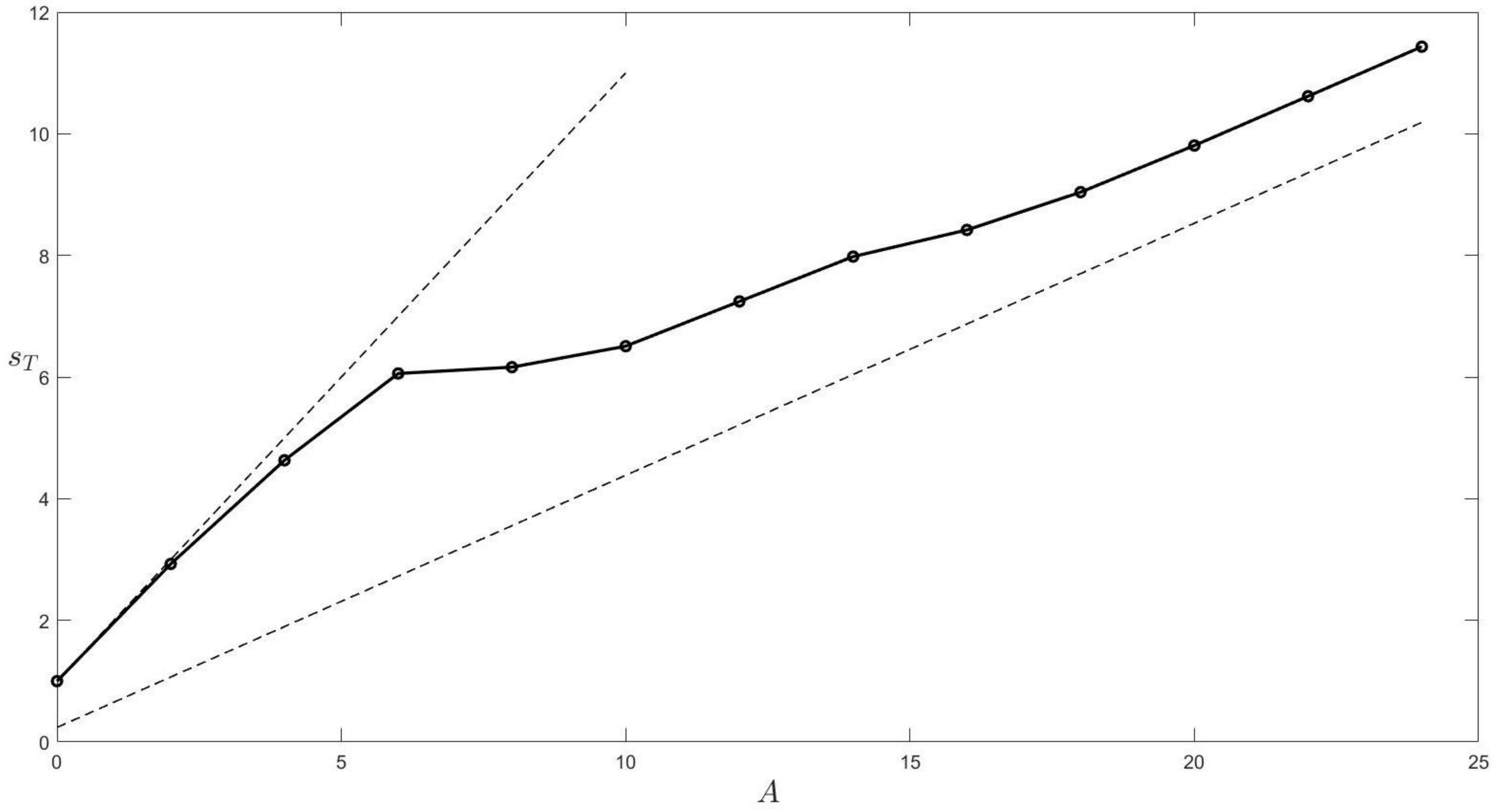}
\caption{Plots of turbulent flame speeds $s_T(A)$ evaluated by (\ref{sTa}). The dash lines refer to the upper and lower bounds in (\ref{sTbd}). Top: ABC flow. Bottom: Kolmogorov flow.}
\label{sTplot}
\end{figure}

\section{Concluding Remarks and Future Works}
\label{sec:5}

While periodic orbits exist for both ABC flow and Kolmogorov flow, they are quite different qualitatively. The ABC flow is known for existence of ``vortex tubes'' as bundled trajectories in axial directions \cite{DFG86}. In fact the positive and negative regions in \cref{Xmap} refer to the cross section of the vortex tubes in $\pm x$-direction respectively. For Kolmogorov flow however, there are two vortex tubes in each direction. Also the vortex tubes of Kolmogorov flow are much ``thinner'' and more ``twisted'' than the vortex tubes of ABC flow. This structural difference contributes to the maximal (sub-maximal) growth of eddy diffusivity in the ABC (Kolmogorov) flow as the molecular diffusivity tends to zero \cite{WXZ}. The streamlines are chaotic outside the vortex tubes, suggesting that the Kolmogorov flow is more disordered. 

Much more remains to be investigated in these two prototypical chaotic flows. For G-equation, we plan to develop a semi-Lagrangian scheme to compute 
the laminar term yet still efficient in computation and accessible with WENO method. Also we shall develope a semi-Lagrangian computation of other 
G-equation models \cite{P00,LXY13} where flame stretching or mean curvature of the level surface appears as a non-constant laminar speed.


\end{document}